\documentclass[reqno,15pt]{amsart}
\usepackage{amsmath}
\usepackage{amssymb}
\usepackage{amscd}
\usepackage{graphics}
\usepackage{latexsym}
\usepackage{hyperref}

\theoremstyle{plain}
\newtheorem{theorem}{Theorem}[section]
\newtheorem{thm}[theorem]{Theorem}
\newtheorem{cor}[theorem]{Corollary}
\newtheorem{lem}[theorem]{Lemma}
\newtheorem{prop}[theorem]{Proposition}

\newcounter{kludge}
\newcounter{kludgeb}
\theoremstyle{definition}

\theoremstyle{remark}

\input xy
\xyoption{all}

\newcommand{\marpar}[1]{}
\newcommand{\mni}{\medskip\noindent}

\newcommand{\mbb}{\mathbb}
\newcommand{\QQ}{\mbb{Q}}

\newcommand{\ZZ}{\mbb{Z}}

\newcommand{\PP}{\mbb{P}}

\newcommand{\mc}{\mathcal}

\newcommand{\OO}{\mc{O}}

\newcommand{\wt}{\widetilde}

\newcommand{\SP}{\text{Spec }}
\newcommand{\Hilb}{\text{Hilb}}
\newcommand{\HB}[2]{\Hilb^{#1}_{#2}}

\newcommand{\n}[1]{n_{#1}}
\newcommand{\nee}{\n{e}}
\newcommand{\N}[1]{N_{#1}}
\newcommand{\Ne}{\N{e}}
\newcommand{\wN}[1]{\wt{N}_{#1}}
\newcommand{\wNe}{\wN{e}}
\newcommand{\p}[1]{P_{#1}}
\newcommand{\pr}{\p{r}}

\newcommand{\f}[1]{f_{#1}}
\newcommand{\fe}{\f{e}}

\newcommand{\F}[1]{F_{#1}}
\newcommand{\Fe}{\F{e}}

\newcommand{\Gg}[1]{G_{#1}}
\newcommand{\Gge}{\Gg{e}}

\newcommand{\lt}{\left}
\newcommand{\rt}{\right}

\newsavebox{\sembox}
\newlength{\semwidth}
\newlength{\boxwidth}

\newsavebox{\semrbox}
\newlength{\semrwidth}
\newlength{\boxrwidth}

\title
{Veronese Varieties Contained in Hypersurfaces} 

\author[Starr]{Jason Michael Starr}
\address{Department of Mathematics \\
  Stony Brook University \\ Stony Brook, NY 11794}
\email{jstarr@math.sunysb.edu}

\date{\today}

\begin{document}


\begin{abstract}
Alex Waldron proved 
that for sufficiently general degree $d$ hypersurfaces in 
projective $n$-space, the Fano scheme parameterizing
$r$-dimensional linear spaces contained in the hypersurface is nonempty
precisely for the degree range $n\geq \N{1}(r,d)$
where
the ``expected dimension'' $\f{1}(n,r,d)$ is nonnegative, in which case
$\f{1}(n,r,d)$ equals the (pure) dimension.
Using work by Gleb Nenashev, we
prove that for sufficiently general 
degree $d$ hypersurfaces in projective $n$-space, the parameter space 
of $r$-dimensional $e$-uple Veronese varieties contained in the
hypersurface 
is nonempty
of pure dimension equal to the ``expected dimension'' $\fe(n,r,d)$
in a degree
range $n\geq \wNe(r,d)$ that is asymptotically sharp.  Moreover, 
we show that for $n\geq 1+\N{1}(r,d)$, the Fano scheme parameterizing
$r$-dimensional linear spaces is irreducible.
\end{abstract}


\maketitle



\section{Introduction and Statement of Results} \label{sec-int}
\marpar{sec-int}

\mni
Let $k$ be an algebraically closed field, not necessarily of
characteristic $0$.  For integers $n,r,e>0$, 
a \emph{Veronese $e$-uple $r$-fold in} $\PP^n_k$ is the image of a
morphism $\nu:\PP^r_k\to \PP^n_k$ such that $\nu^*\OO(1)$ is
isomorphic to $\OO(e)$ and such that the pullback homomorphism,
$$
\nu^*_1:H^0(\PP^n_k,\OO(1))\to H^0(\PP^r_k,\OO(e)),
$$
is surjective; such a morphism is a closed immersion.  
For brevity, the image $\nu(\PP^r)$ of such a morphism is called a
$V_e^r$.  
Denoting by $\pr(t) \in \QQ[t]$ the numerical
polynomial with $\pr(d) = \binom{d+r}{r}$ for all integers $d\geq -r$,
the $\OO(1)$-Hilbert polynomial of the image of $\nu$ equals
$\pr(et)$.  Denote $\pr(e)-1$ by $\nee(r)$, e.g., $n_1(r)$ equals $r$, and
$n_2(r)$ equals $r(r+3)/2$.  
In the Hilbert scheme $\text{Hilb}^{\pr(et)}_{\PP^n/k}$
there is an open subscheme $\Gge(r,\PP^n_k)$ parameterizing Veronese 
$e$-uple $r$-folds.  This open scheme is nonempty precisely when
$h^0(\PP^n_k,\OO(1)) \geq h^0(\PP^r,\OO(e))$, i.e., when $n\geq
\nee(r).$  
When it is nonempty, $\Gge(r,\PP^n_k)$ 
is smooth and geometrically integral of
dimension 
$$
f_e(n,r):= (n+1)(\nee(r)+1)-(r+1)^2.
$$
Note, in particular, that $f_1(n,r)$ equals $(n-r)(r+1)$, which is
nonnegative if and only if $n\geq r$, i.e., $n\geq n_1(r)$.  Please note,
for every $e\geq 2$, for every $n\geq 1$, for every $r\geq 1$,
$f_e(n,r)$ is positive, even though $\Gge(r,\PP^n_k)$ is empty for
$n<\nee(r)$.  
In fact, $\Gge(r,\PP^n_k)$ has a natural action of $\textbf{PGL}_{n+1}$
under which it is smoothly homogeneous (the stabilizer subgroup is reduced).
Assuming that it is nonempty, the
quasi-projective scheme $\Gge(r,\PP^n_k)$ is projective precisely when
$r$ equals 
$1$.  For $r=1$, this is a classical Grassmannian,
$\Gg{1}(r,\PP^n_k)=\text{Grass}(r,\PP^n_k)$, 
parameterizing
$r$-dimensional projective linear subspaces of $\PP^n_k$.  

\mni
For every (locally) closed subscheme $X\subset \PP^n_k$, the \emph{Fano
  scheme} of Veronese $e$-uple $r$-folds in $X$ is the intersection
$\Fe(r,X)$ of the open subscheme $\Gge(r,\PP^n_k)$ with the (locally)
closed subscheme $\text{Hilb}^{\pr(et)}_{X/k}$ in
$\text{Hilb}^{\pr(et)}_{\PP^n/k}$.  Even though the $V^r_e$ varieties
are among the simplest $r$-cycles in $\PP^n_k$, little is known
about $\Fe(r,X)$ for general $X$.  Here we study $\Fe(r,X)$ for
degree $d$ hypersurfaces $X$ in $\PP^n_k$.  In that case, there is a
lower bound on the dimension of every (nonempty) irreducible
component as follows,
$$
f_e(n,r,d) = f_e(n,r) - \pr(ed) = (n+1)\pr(e) - \pr(de)-(r+1)^2.
$$
Denote by $\Ne(r,d)$ the smallest integer $n\geq n_r(e)$ such that
$f_e(n,r,d)\geq 0$, i.e.,
$$
\Ne(r,d) = \max\lt(-1+\pr(e),-1+\lt\lceil \frac{(r+1)^2+\pr(ed)}{\pr(e)}
\rt\rceil \rt).
$$
It is occasionally convenient to use $M_e(r,d) = \Ne(r,d) - \pr(e)+1$,
i.e.,
$$
M_e(r,d) = \max\lt(0, \lt\lceil \frac{(r+1)^2 + \pr(ed)-\pr(e)^2}{\pr(e)}
\rt\rceil \rt).
$$
Recall that the complete linear system $\PP H^0(\PP^n_k,\OO(d))$
parameterizing degree $d$ hypersurfaces in $\PP^n_k$ is
canonically $k$-isomorphic to the Hilbert scheme
$\text{Hilb}^{P_n(t)-P_n(t-d)}_{\PP^n/k}$ of closed subschemes of
$\PP^n_k$.  Denote by $\mc{X}\subset \PP H^0(\PP^n_k,\OO(d))
\times_{\SP(k)}\PP^n_k$ the universal closed subscheme such that the
restricted projection $k$-morphism is flat,
$$
\pi:\mc{X} \to \PP H^0(\PP^n_k,\OO(d)).
$$
Denote by
$\Fe(r,\mc{X}) \subset \PP H^0(\PP^n_k,\OO(d)) \times_{\SP(k)}
\Gge(r,\PP^n_k)$ the relative Fano scheme of Veronese $e$-uple
$r$-folds in fibers of $\pi$, i.e., the intersection of the open
subscheme $\PP H^0(\PP^n_k,\OO(d))\times_{\SP(k)} \Fe(r,\PP^n_k)$ with
the relative Hilbert scheme $\text{Hilb}^{\pr(et)}_{\pi}$.  The
two projection morphisms restrict to morphisms on $\Fe(r,\mc{X})$,
$$
\rho:\Fe(r,\mc{X}) \to \Gge(r,\PP^n_k),
$$
$$
\pi:\Fe(r,\mc{X}) \to \PP H^0(\PP^n_k,\OO(d)).
$$
Denote by $\Fe(r,\mc{X})_{\text{fl}}$,
resp. $\Fe(r,\mc{X})_{\text{sm}}$, the maximal open subscheme of
$\Fe(r,\mc{X})$ on which $\pi$ is flat, resp. smooth.
Denote by
$U_{e,d}^r \subset \PP H^0(\PP^n_k,\OO(d))$, resp. $V_{e,d}^r\subset
U_{e,d}^r$, 
the maximal open subscheme
over which $\pi$ is flat, resp. smooth.  The open $U_{e,d}^r$ 
is a dense open by the Generic
Flatness Theorem.  

\begin{prop} \label{prop-inc} \marpar{prop-inc}
For all $n\geq \nee(r)$, for all $d\geq 1$, 
the morphism $\rho$ is a Zariski locally
trivial projective bundle of relative dimension $P_n(d) - \pr(de)$.
Thus the restriction of $\pi$ over $U_{e,d}^r$ either has empty fibers
or else it is flat of relative dimension $f_e(n,r,d)$.  
If the characteristic equals $0$ or $p\geq
p_e(n,r,d)$ for an effectively computable integer $p_e(n,r,d)$, then
$V_{e,d}^r$ is a dense open subset of $U_{e,d}^r$, i.e., $\pi$ is
smooth over a dense open of $U_{e,d}^r$.  
\end{prop}

\mni
By the proposition, in order that $\Fe(r,X)$ is nonempty for every
degree $d$ hypersurface $X$, it is necessary that $n\geq n_r(e)$ and
$f_e(n,r,d)\geq 0$.
This is equivalent to the condition that $n\geq \Ne(r,d)$.  Also, for
the difference $m:= n- \nee(r)$, it is equivalent to the condition that
$m\geq M_e(r,d)$.  Using a theorem of Hochster-Laksov, Alex Waldron
proved that the necessary condition is sufficient for $e=1$, i.e., for
linear spaces.


\begin{thm}[Waldron, \cite{Waldron}] \label{thm-Waldron} \marpar{thm-Waldron}
For all $d\geq 3$, 
for all $n \geq \N{1}(r,d) = r+\lceil \pr(d)/(r+1) \rceil$, the
smooth locus of $\pi$ in $\F{1}(r,\mc{X})$ is dense, thus, for every
degree $d$ hypersurface $X$ in $\PP^n_k$, the Fano scheme $\F{1}(r,X)$
of linear $r$-planes in $X$ is nonempty.  For $d=1$, this is true for
all $n\geq \N{1}(r,1) = 1+\n{1}(r) = 1+r$.  For $d=2$, this is true
precisely for $n\geq 1+2r.$
\end{thm}

\mni
Please note, when $d=2$, then $1+2r > \N{1}(r,2)$ for all $r\geq 2$, so
this case is special.  
Recently, Gleb Nenashev has generalized the Hochster-Laksov theorem.
Using this generalization, there is a similar result for all $e\geq 2$ for a
bound $n\geq \wNe(r,d)$, where $\wNe(r,d)$ 
is asymptotically sharp for fixed $e$, $r$, and increasing
$d$. 

\begin{thm} \label{thm-Veronese} \marpar{thm-Veronese}
For $e\geq 2$, for all $n$ at least
$\wNe(r,d) := -1 + 2\pr(e) + \lceil \pr(de)/\pr(e) \rceil$, the
smooth locus of $\pi$ in $\Fe(r,\mc{X})$ is dense. Thus, for every
sufficiently general
degree $d$ hypersurface $X$ in $\PP^n_k$, the Fano scheme $\Fe(r,X)$
of Veronese $e$-uple $r$-folds in $X$ is a nonempty, geometrically
reduced, local complete intersection scheme of dimension
$f_e(r,d,n)$. 
For $d=1$, this is true precisely for $n\geq 1+\nee(r)$.  For $d=2$,
this is true precisely for $n\geq \nee(r) = \Ne(r,2)$.
\end{thm}

\mni  
What about irreducibility,
i.e., connectedness?  The method we use to study this, based on
Minoccheri's form of Bertini's irreducibility theorem,
cf. \cite{Minoccheri}, uses 
projective parameter spaces.  So the result works best for linear
spaces, 
$e=1$.  The integer $\Ne'(r,d)$ is
the least integer $n$ such that the complement of the smooth locus
$\Fe(r,\mc{X})_{\text{sm}}$ of $\pi$ in $\Fe(r,\mc{X})$ has
codimension $\geq 2$ everywhere.

\begin{prop} \label{prop-dim} \marpar{prop-dim}
If $n\geq \N{1}'(r,d)$, 
then for every degree $d$ hypersurface $X$ in $\PP^n_k$, 
the Fano scheme $\F{1}(r,X)$ is geometrically connected.  For
sufficiently general $X$, the Fano scheme is a local complete
intersection scheme that is geometrically integral
and normal.  If the characteristic is $0$ or $p>p_1(n,r,d)$, then for
sufficiently general $X$, the Fano scheme is also smooth.
\end{prop}

\mni
This connectedness result for linear spaces implies connectedness
results for more general cycles.  The following corollary is one
example of this; certainly the bound can be improved.

\begin{cor} \label{cor-dim} \marpar{cor-dim}
If $n\geq \N{1}'(d,\nee(r))$, then there exists a dense, Zariski open
subscheme $W_{e,d}^r\subset U_{e,d}^r$ such that 
for every degree $d$ hypersurface $X$ with $[X]\in W_{e,d}^r$, 
the Fano scheme $\Fe(r,X)$ is geometrically connected.  
If the characteristic is $0$ or $p>p_1(n,r,d)$, then for
sufficiently general $X$, the Fano scheme is also smooth.
\end{cor}

\mni
The following bound for $\N{1}'(r,d)$ is sharp to within $1$.  There are
infinitely many cases when the bound is sharp.

\begin{thm} \label{thm-nprime} \marpar{thm-nprime}
Regarding $\N{1}'(r,d)$, we have the following.
\begin{enumerate}
\item[(i)]
For all $d\geq 2$, $\N{1}(r,d)\leq \N{1}'(r,d)\leq
1+\N{1}(r,d)$, so the Fano schemes $\F{1}(r,x)$
are geometrically connected if $n\geq
1 + \N{1}(r,d)$.  
\item[(ii)]
For $d=1$, $\Ne(r,1)$ equals
$1+\nee(r)$, and the Fano schemes $\F{1}(r,X)$ are
connected for all $n\geq \nee(r)$.  
\item[(iii)]
For $d=2$, $\N{1}(r,2)$ equals
$2r+1$, $\N{1}'(r,2)$ equals $1+\N{1}(r,2)=2r+2$, and the Fano schemes
$\F{1}(r,X)$
are
geometrically \emph{disconnected} for
$n=\N{1}(r,2)=2r+1$.  
\item[(iv)]
For all $d\geq 2$, if 
$f_1(\N{1}(r,d),r,d)$ equals $0$, then $\N{1}'(r,d)$ equals
$1+\N{1}(r,d)$; in fact, the length of the finite 
Fano scheme for $n=\N{1}(r,d)$ is divisible by $d^{r+1}$.
\end{enumerate}  
\end{thm}

\mni
The case of linear spaces deserves special attention.
In characteristic $0$, for $n\geq \N{1}'(r,d)$, for $X$ sufficiently
general, $\F{1}(r,X)$ is smooth and geometrically irreducible of the
expected dimension.  What can we say for \emph{every} smooth
hypersurface $X$?  In an appropriate degree range, there exists a
canonically defined (nonempty)
irreducible component of $\F{1}(r,X)$ of the expected dimension such
that $\F{1}(r,X)$ is reduced at the generic point of this component. 
It is convenient to introduce the flag
Fano scheme.  

\mni
For every scheme $S$ over $\SP(\QQ)$, for every
$\PP^n$-bundle $\pi:\PP_S(E)\to S$ together with an ample
invertible sheaf $q:\pi^*E^\vee \twoheadrightarrow \OO_E(1)$, for
every locally
closed subscheme $X\subset \PP^n_k$ such that $\pi:X\to S$ is locally
finitely presented, 
denote by $\F{1}(0,1,\dots,r,X/S)$ the flag Hilbert scheme of
$X$, $\text{fHilb}^{P_0(t),P_1(t),\dots,\pr(t)}_{X/S}$, parameterizing
flags of linear subspaces contained in fibers of $X$.  There is a forgetful
$S$-morphism,
$$
\rho_r:\F{1}(0,1,\dots,r-1,r,X/S)\to \F{1}(0,1,\dots,r-1,X/S).
$$
Assume now that $X$ is $S$-smooth.  Then for  
every component of $\F{1}(0,1,\dots,r,X/S)$ parameterizing flags of linear
subspace $\Lambda_0\subset \Lambda_1\subset \dots\subset\Lambda_r$ in
geometric fibers of $X/S$, there
is a lower bound $e_r(X/S)$ on the dimension of every
irreducible component of every (nonempty) fiber of $\rho_r$,
$$
e_r(X/S)= -r-1+\sum_{\ell=1}^rb_{r,\ell}\langle
\text{ch}_\ell(T_{X/S}),[\Lambda_\ell] \rangle, 
$$
where $\text{ch}_\ell(T_{X/S})$ is the graded piece of the Chern
character of homogeneous degree $\ell$ of
$T_{X/S}=(\Omega_{X/S})^\vee$, 
and where the rational numbers
$b_{r,\ell}$ are determined by 
$$
\pr(t-1) = \sum_{\ell=1}^r \frac{b_{r,\ell}}{\ell!} t^\ell.
$$
There is a natural infinitesimal deformation theory and obstruction
theory for $\rho_r$.  When the obstruction group vanishes, then
$\rho_r$ is smooth of relative dimension $e_r(X/S)$.  
There exists a sequence $(U_\ell)_{0\leq \ell\leq r-1}$ of open subschemes
$U_\ell\subset \F{1}(0,1,\dots,\ell,X/S)$ such that
\begin{enumerate}
\item[(i)] $U_0\subset X$ is the maximal open subscheme over which
  $\rho_1$ has vanishing obstruction groups so that $\rho_1$ is
  smooth of relative dimension $e_1(X/S)$ over $U_0$,
\item[(ii)] for every $\ell=1,\dots,r-1$, $U_\ell\subset
  \rho_\ell^{-1}(U_\ell)$ is the maximal open subscheme over which
  $\rho_{\ell+1}$ has vanishing obstruction groups so that
  $\rho_{\ell+1}$ is smooth of relative dimension $e_{\ell+1}(X/S)$
  over $U_\ell$.
\end{enumerate}
This sequence is compatible with arbitrary base change over $S$.  The
main result of \cite{Slinear} is the following.

\begin{prop}\cite{Slinear} \label{prop-upper} \marpar{prop-upper}
Assume that $k$ has characteristic $0$.  Let $X\subset \PP^n_k$ be a
smooth hypersurface of degree $d$.
\begin{enumerate}
\item[(i)]
If $n <
r+\pr(d-1)$, then 
$\rho_r^{-1}(U_{r-1})$ is empty.  
\item[(ii)]
If $n\geq r+ \pr(d-1)$, then 
$\rho_r^{-1}(U_{r-1})$ is nonempty, the
restriction of $\rho_r$ over $U_{r-1}$ is smooth and projective of
relative dimension $n-r-\pr(d-1)$, and each geometric fiber is a 
complete intersection in a projective space of a sequence of 
hypersurfaces whose
maximal degree equals $d$.  
\item[(iii)]
If $n$ equals $r+\pr(d-1)$ and $d>1$, 
then the 
fibers of
$\rho_r$ over $U_{r-1}$ are not geometrically connected. 
\item[(iv)]
If $n\geq 1+r+\pr(d-1)$, then every
geometric fiber of $\rho_r$ over $U_r$ is geometrically connected so
that $\rho_r^{-1}(U_r)$ is smooth and irreducible.
\end{enumerate} 
\end{prop}

\mni
\textbf{Acknowledgments.}  I am grateful to Joe Harris for asking
about connectedness for Fano schemes of hypersurfaces, which led to
this paper.  This work was supported by NSF Grant DMS-1405709.


\section{Proof of Proposition \ref{prop-inc}} \label{sec-inc}
\marpar{sec-inc}

\mni
This proposition follows by the general method of incidence correspondences.  
Let $S$ be a scheme.  Let $E$ be a locally free $\OO_S$-module of rank $n+1$.  
Let $\pi_0:\PP_S(E)\to S$ together with $\pi_0^*E^\vee\to \OO_E(1)$
represent the functor that associates to every $S$-scheme, $f:T\to S$,
the set of invertible sheaf quotients on $T$ of $f^*E^\vee$.  Then
$\pi$ is a $\mathbb{P}^n$-bundle over $S$.

\mni
For every locally closed subscheme $X\subset \PP_S(E)$ that is locally
finitely presented over $S$ (automatic for Noetherian schemes), for
every integer $r$,
the
associated \emph{Fano scheme}, $\F{1}(r,X/S)$, is the Hilbert scheme
$\text{Hilb}^{\pr(t)}_{X/S}$.  
Hilbert polynomials are with respect to
the invertible sheaf $\mathcal{O}_E(1)$.  
Of course
$\F{1}(r,\mathbb{P}(E)/S)$ is the Grassmannian bundle associated to $E$,
i.e., an $S$-scheme
$\pi_r:\text{Grass}_S(r+1,E)\to S$ together with the locally free
quotient $\pi_r^*E^\vee \to Q_{E,r}$ of rank $r+1$ that represents the
functor sending $f:T\to S$ to the set of rank $r+1$  
locally free quotients  of $f^*E^\vee$.  

\mni
For every integer $d\geq 0$, denote by $S_d(E)$ the locally free sheaf
$(\pi_0)_*\OO_E(d)$, so that the direct sum
$S(E) := \bigoplus_{d\geq 0} S_d(E)$ with its natural product structure is the
homogeneous coordinate ring of $\PP_S(E)$ with respect to $\OO_E(1)$.
In other words, $S_d(E)$ is the degree $d$ symmetric power of
$E^\vee$, i.e., for the tensor algebra $T(E)$ of $E^\vee$, the algebra
quotient $T(E)\to S(E)$ is initial among morphisms of sheaves of
associative $\mathcal{O}_S$-algebras that are commutative.
Since the tensor algebra of the invertible sheaf $\mathcal{O}_E(1)$ on 
$\PP_S(E)$ is commutative, the invertible quotient 
$\pi_0^*S_1(E) \to \OO_V(1)$
induces an invertible quotient $\pi_0^*S_d(E) \to \OO_V(d)$.  Thus, on
the fiber product $\PP_S(S_d(E))\times_S\PP_S(E)$, there is a natural
morphism of invertible sheaves, 
$$
\alpha:\text{pr}_1^*\OO_{S_d(E)}(-1) \to \text{pr}_2^* \OO_E(d).  
$$
The support of the cokernel of $\alpha$ is a Cartier divisor
$\mc{X}\subset \PP_S(S_d(E))\times_S\PP_S(E)$ that is flat with
respect to $\text{pr}_1$ and has relative degree $d$ with respect to
$\text{pr}_2^*\OO_E(1)$. 

\mni
For every separated, finitely presented morphism, $\pi:Z\to S$, for
every quasi-coherent $\OO_Z$-module $\mc{E}$ that is locally finitely
presented, that is $\OO_S$-flat, and 
that has proper support over $S$, there is
a maximal open subscheme $U=U_{\pi,\mc{E}}\subset S$ such that the
complement of $U$ equals the (locally finite) 
union of the supports of $R^q\pi_*\mc{E}$
for the (locally finitely) many $q>0$ such that $R^q\pi_*\mc{E}$ is
nonzero.  By \cite[Corollaire 7.9.10, Lemme 7.9.10.1]{EGA3}, 
for every $S$-scheme $f:T\to S$, for the base change $\pi_T:Z_T\to T$
of $\pi$, and for the pullback $\mc{E}_T$ of $\mc{E}$ to $Z_T$, the
open subset $U_{\pi_T,\mc{E}_T} \subset T$ equals
$f^{-1}U_{\pi,\mc{E}}$, $(\pi_T)_*\mc{E}_T$ is a locally free
$\OO_T$-module of (locally) finite rank, and the natural map
$f^*\pi_*\mc{E}\to (\pi_T)_*\mc{E}_T$ is an isomorphism.  

\mni
In particular, for a numerical polynomial $P(t)$, for the Hilbert
scheme $\HB{P(t)}{\PP(E)/S}$ with its universal closed subscheme
$Z\subset \HB{P(t)}{\PP(E)/S}\times_S \PP_S(E)$ with its projections
$$
\pi_{E,P(t)}:Z\to \HB{P(t)}{\PP(E)/S},
$$ 
$$
\rho_{E,P(t)}:Z\to \PP_S(E)
$$
for every integer $d\geq 1$, there
exists a maximal open subscheme $U_d \subset \HB{P(t)}{\PP(E)/S}$ such
that $R^q\pi_*\rho^*\OO_{E}(d)$ equals zero on $U_d$ for all $q>0$.  On
this open subset, the sheaf $\pi_*\rho^*\OO_E(d)$ is locally free.
There is a natural base change homomorphism of $\OO_{U_d}$-modules,
$$
\phi_{d}: S_d(E)\otimes_{\OO_S} \OO_{U_d} \to \pi_*\rho^*\OO_E(d)|_U.
$$
Denote by $V_d\subset U_d$ the maximal open subscheme on which
$\phi_d$ is surjective, i.e., $V_d$ is the relative complement in
$U_d$ of the support of the cokernel of $\phi_d$.  In this case, the
kernel $\mc{K}_d$ of $\phi_d$ on $V_d$ is locally free.  Thus the dual of the
kernel, $\mc{K}^\vee_d$, is also locally free on $V_d$.  Denote by
$\kappa:\PP_{V_d}(\mc{K}_d) \to V_d$ the associated projective bundle
with its universal invertible quotient $\kappa^*\mc{K}_d^\vee \to
\OO_{\mc{K}_d}(1)$.  

\mni
Since
$\pi_*\rho^*\OO_E(d)|_U$ is locally free on $V_d$, the associated
$\OO_{V_d}$-module homomorphism
$$
\psi_d:S_d(E)^\vee\otimes_{\OO_S}\OO_{V_d}\to \mc{K}_d^\vee
$$
is surjective.  
Thus, there is a unique $S$-morphism, $\iota:\PP_{V_d}(\mc{K}_d) \to
\PP_S(S_d(E))$, such that $\iota^*\OO_{S_d(E)}(1)$ equals
$\OO_{K_d}(1)$ and such that $\psi_d$ is the induced homomorphism on
global sections of $\OO_{S_d(e)}(1)$, resp. $\OO_{K_d}(1)$.  

\mni
In the special case that the Hilbert polynomial $P(t)$ equals
$P_{n,d}(t) = P_n(t) - P_n(t-d)$, this gives the following.

\begin{lem} \label{lem-Hilb} \marpar{lem-Hilb}
The closed subscheme $\mc{X}\subset \PP_S(S_d(E))\times_S\PP_S(E)$
defines an 
isomorphism from $\PP_S(S_d(E))$ to the Hilbert scheme
$\text{Hilb}^{P_{n,d}(t)}_{\PP_S(E)/S}$. 
\end{lem}

\begin{proof}
This is well-known.  
Here is the basic idea.  First of all, since
$\mc{X}$ is a Cartier divisor in a scheme that is flat over
$\PP_S(S_d(E))$, the Cartier divisor is flat over $\PP_S(S_d(E))$ if
and only if every geometric fiber is a Cartier divisor in the
geometric fiber of $\PP_S(E)$.  This is true since $\alpha$ is nonzero
on geometric fibers.  Thus, there is an induced morphism from
$\PP_S(S_d(E))$ to the Hilbert scheme.

\mni
By the computation of
cohomology of invertible sheaves on projective space, the image of
$\PP_S(S_d(E))$ maps into the open subset $U_e$ of the Hilbert scheme
for every integer $e\geq d-n$.  By computation on geometric points of
$\PP_S(S_d(E))$, the pullback of $\phi_e$ is surjective for every
$e\geq d-n$.  Thus, the morphism to the Hilbert scheme factors through
the open subset $V_e$.  On the other hand, on the open subset $V_d$,
since $P_{n,d}(d)$ equals $P_n(d)-1$, $\psi_d$ is an invertible
quotient of the pullback of $S_d(E)^\vee$.  This invertible quotient
defines an inverse morphism from $V_d$ to $\PP_S(S_d(E))$.  

\mni
Finally, to prove that $V_d$ equals the entire Hilbert scheme, it
suffices to compute on geometric points $\SP(k)\to S$.  
For a closed subscheme
$Z\subset \PP_k(E_k)$ with Hilbert polynomial $P_{n,d}(t)$, since the
degree of Hilbert polynomial equals $n-1$, there are associated primes
of $Z$ of height $1$, and every such prime is minimal.  The
intersection of the finitely many primary components of $\OO_Z$ for
such primes gives an ideal sheaf whose associated closed scheme
$Z^{(1)}$ is a Cartier divisor in $\PP_k(E_k)$ contained in $Z$ and
that equals the divisorial part of $Z$.  Since the leading coefficient
of $P_{n,d}(t)$ equals $d/(n-1)!$, $Z^{(1)}$ has degree $d$.  As a
degree $d$ hypersurface in $\PP_k(E_k)$, the Hilbert polynomial of
$Z^{(1)}$ equals $P_{n,d}(t)$.  Thus, for the natural surjection
$\OO_Z\to \OO_{Z^{(1)}}$, the kernel has Hilbert polynomial zero.
Thus the kernel is zero, i.e., $Z$ equals the degree $d$ hypersurface
$Z^{(1)}$.  
\end{proof}

\mni
Returning to the case of an arbitrary Hilbert polynomial $P(t)$, we have
the following generalization.

\begin{prop} \label{prop-opend} \marpar{prop-opend}
Inside $\PP_{V_d}(\mc{K}_d)\times_S \PP_S(E),$ the closed subscheme
$(\kappa\times \text{Id}_{\PP(E)})^{-1}\mc{Z}$ is contained in the
closed subscheme $(\iota\times \text{Id})^{-1}\mc{X}$.  Associated to
this pair of closed subschemes, flat over $\PP_{V_d}(\mc{K}_d)$, the
induced morphism from $\PP_{V_d}(\mc{K}_d)$ to the flag Hilbert scheme
$\text{fHilb}^{P(t),P_{d,n}(t)}_{\PP(E)/S}$ is an open immersion whose open
image equals the inverse image of $V_d$ via the forgetful morphism
$\Phi_{P(t),d}:\text{fHilb}^{P(t),P_{d,n}(t)}_{\PP(E)/S} \to \HB{P(t)}{\PP(E)/S}$.  
\end{prop}

\begin{proof}
By construction, on $\PP_{V_d}(\mc{K}_d)$, the defining polynomials of
$\mc{X}$, considered as sections of $S_d(E)$, vanish when restricted
to $\pi_*\rho^*\OO_E(d)|_{V_d}$.  Thus the pullback of $\mc{Z}$ is
contained in the pullback of $\mc{X}$.  Thus, there is an induced
morphism to the flag Hilbert scheme.  By construction, the image of
this morphism is contained in the inverse image of $V_d$.  Now we
repeat the argument in the previous lemma to construct an inverse
isomorphism from the inverse image of $V_d$ to $\PP_{V_d}(\mc{K}_d)$.  
\end{proof}

\mni
Since $\mc{K}_d$ is locally free of rank $P_n(d)-P(d)$, the projection
$\PP_{V_d}(\mc{K}_d)\to V_d$ is smooth of relative dimension $P_n(d)-P(d)$.
Thus, we have a corollary of the previous proposition.

\begin{cor} \label{cor-opend} \marpar{cor-opend}
The forgetful morphism $\Phi_{P(t),d}:\Phi_{P(t),d}^{-1}(V_d)\to V_d$
is smooth, even a projective bundle, of relative dimension $P_n(d)-P(d)$.  
\end{cor}

\mni
Using the corollary, the first part of the proposition is reduced to
the following result. 

\begin{lem} \label{lem-VeroneseV} \marpar{lem-VeroneseV}
The open subscheme $\Gge(r,\PP(E))$ of the Hilbert scheme is contained
in the open subscheme $V_d$ for every integer $d\geq 1$.
\end{lem}

\begin{proof}
Since this is a statement about equality of two open subsets, this can
be checked at the level of geometric points of the Hilbert scheme.
Thus, assume that $k$ is algebraically closed, and let $\nu:\PP_k(E_r)
\to \PP_k(E)$ be a Veronese $e$-uple morphism.
For every integer $d\geq 1$, $\nu^*\OO_E(d)$
equals $\OO_{E_r}(de)$.  By the computation of cohomology of line bundles on
projective space, $h^q(\PP_k(E_r),\OO(de))$ is zero for all $q>0$ and for
all $d\geq 1$.  Thus, $\text{Image}(\nu)$ gives a point of $U_d$.
Finally, by hypothesis, 
$$
\nu^*_1:H^0(\PP(E)_k,\OO_E(1))\to H^0(\PP(E_r)_k,\OO_{E_r}(e)),
$$
is surjective.  The induced map $\phi_d$ is just the composite of the 
$d^{\text{th}}$
symmetric power of $\nu^*_1$ and the evaluation morphism,
$$
\text{Sym}^d_k H^0(\PP_k(E),\OO_E(1))\to \text{Sym}^d_k
H^0(\PP_k(E_r),\OO_{E_r}(e)) \to H^0(\PP_k(E_r),\OO_{E_r}(de)).
$$
The first factor is surjective by hypothesis, and the second factor is
surjective by the computation of cohomology of line bundles on
projective space.  Thus, $\text{Image}(\nu)$ is a point of $V_d$.  
\end{proof}

\mni
As a special case of the lemma that will be useful later, the $d=1$
result implies that the following is a short exact sequence of locally
free sheaves on $\Fe(r,\PP(E))$ compatible with arbitrary base change,
$$
0 \to \mc{K}_1 \to E^\vee\otimes_{k} \OO_{\Gge(r,\PP(E))} \to
\pi_*\rho^*\OO_E(1) \to 0.
$$
Denote the quotient by $E^\vee_G$.  Then $\PP_G(E_G)\to
\Gge(r,\PP_k(E))$ 
is a projective 
subbundle of the projective bundle
$\Gge(r,\PP_k(E))\times_{\SP(k)}\PP_k(E)$ that is flat over
$\Gge(r,\PP_k(E))$ of relative dimension $\nee(r)$.  By construction,
$\PP_G(E_G)$ contains the restriction over the open $\Gge(r,\PP_k(E))$ of the
universal closed subscheme over the entire Hilbert scheme.  
Thus, this pair of
closed subschemes gives a morphism to the flag Hilbert scheme,
$$
\Psi_{e,r,E}:\Gge(r,\PP_k(E))\to
\text{fHilb}^{\pr(et),P_{\nee(r)}(t)}_{\PP(E)/k}.  
$$
On the other hand, there is a forgetful morphism,
$$
\Phi:
\text{fHilb}^{\pr(et),P_{\nee(r)}(t)}_{\PP(E)/k}
\to
\text{Hilb}^{\pr(et)}_{\PP(E)/k}.
$$
By construction, $\Phi \circ \Psi$ is the inclusion, so that the image
of $\Psi$ is contained in $\Phi^{-1}(\Gge(r,\PP_k(E)))$.  Altogether,
this proves the following.

\begin{cor} \label{cor-VeroneseV} \marpar{cor-VeroneseV}
The forgetful morphism $\Phi:\Phi^{-1}(\Gge(r,\PP_k(E)))\to
\Gge(r,\PP_k(E))$ is an isomorphism, and the pullback via the inverse
isomorphism $\Psi$ of the 
universal linear $\nee(r)$-fold containing the Veronese is $\PP(E_G)$,
the family of linear spans of the Veronese $e$-uple $r$-folds.
\end{cor}

\mni
The last part of Proposition \ref{prop-inc} follows in characteristic
$0$ 
by Generic
Smoothness.  Of course the characteristic $0$ result implies that
there exists some integer $p_e(n,r,d)$ such that the result also holds
whenever the characteristic $p$ satisfies $p\geq p_e(n,r,d)$.  In
fact, this integer is effectively computable, even though the
effective upper bounds here are probably far from optimal.  The key is
the following observation.

\begin{lem} \label{lem-insep} \marpar{lem-insep}
Let $k$ be a field.  Let $S$ and $T$ be smooth, integral $k$-schemes.
Let $f:S\to T$ be a dominant morphism.  For every irreducible
component $B$ of the singular locus of $f$ (defined via Fitting ideals
of $\Omega_f$) endowed with its induced reduced structure, 
if $B$ dominates $T$, then $B\to T$ is not separable.
\end{lem}

\begin{proof}
Denote by $S^o\subset S$, resp.
$B^o\subset B$, the $k$-smooth locus of $f$, resp. of $f|_B$.
Then $(df)^\dagger:f^*\Omega_{T/k}\to \Omega_{S/k}$, resp.
$d(f|_B)^\dagger:(f|_B)^*\Omega_{T/k}\to \Omega_{B/k}$, is a local split injection
with locally free cokernel
on $S^o$, resp. $B^o$.  Since $d(f|_B)^\dagger$
factors through $\Omega_{S/k}|_B\to
\Omega_{B/k}$, it follows that $B^o$ is contained in $S^o\cap B$.   
Since $B$ is
disjoint from $S^o$, $B^o$ is empty.  Therefore $f|_B$ is not
separable.
\end{proof}

\mni
In case $S$ is a specific quasi-projective $T$-scheme, up to
intersecting $S$ with a sufficiently general collection of hyperplane
sections, it suffices to assume that $B$ is generically finite over
$T$.  Then, since $B\to T$ is not separable, the length of
$\OO_{B,\eta}$ as an $\OO_{T,\eta}$-module, $\eta$ a generic point of
$T$, is at least $p$.  On the other hand, there are upper bounds on
the length of the singular locus of the zero-dimensional components of
the singular locus of $S$ in terms of the dimension and degree of $S$, cf.
\cite[Section 4.2]{JanGutt}.  Using this, it is possible to find an
effective upper 
bound on $p_e(n,r,d)$ in terms of dimensions and degrees of Hilbert
schemes.    


\section{Proof of Theorem \ref{thm-Veronese}} \label{sec-Veronese}
\marpar{sec-Veronese}

\mni
The proof of the main part of the theorem is very similar to the proof
of the theorem of Alex Waldron \cite{Waldron}.

\mni
Since smoothness can be checked after base
change from $k$ to an algebraic closure, assume that $k$ is
algebraically closed.  As above, assume that $E$ is a $k$-vector space
of rank $n+1$ so that $(\PP_k(E),\OO_E(1))$ is $k$-isomorphic to
$\PP^n_k$ with its Serre twisting sheaf.

\mni
Let $E_r$ be a $k$-subspace of rank
$r+1$, and let $\nu:\PP_k(E_r)\hookrightarrow \PP_k(E)$ denote a
Veronese $e$-uple morphism.
Denote by $\mc{J}$ the
corresponding ideal sheaf.
In particular, the $k$-subspace $J_1:=H^0(\PP_k(E),\mc{J}(1))$ of
$E^\vee = H^0(\PP_k(E),\OO_{E}(1))$ equals the kernel of the
surjection, 
$$
\nu^*_1: H^0(\PP_k(E),\OO_{E}(1)) \to H^0(\PP_k(E_r),\OO_{E_r}(e)).
$$
This is the same as the pullback of the sheaf $\mc{K}_1$ from the
previous section.
The annihilator of $J_1$ is a linear subspace $E_\nu\subset E$ of
dimension $\pr(e)$, the pullback of $E_G$ from the previous section.  
The subvariety $\PP_k(E_\nu) = \text{Zero}(J_1)$
of $\PP_k(E)$ is the unique linear subvariety of dimension $\nee(r)$ 
that contains the image
of $\nu$, i.e., $\PP_k(E_\nu)$ is the linear span of $\nu$.  In
particular, for the ideal sheaf $\mc{J}_1$ of $\PP_k(E_\nu)$ in $\PP_k(E)$,
$\mc{J}_1|_{\PP_k(E_\nu)}$ equals $J_1\otimes_k\OO_{E_\nu}(-1)$.  

\mni
The fundamental exact sequence of sheaves of relative differentials
is,
$$
\begin{CD}
0 @>>> \nu^*\mc{J} @> \delta >> \nu^*\Omega_{\PP(E)/k} @> (d\nu)^\dagger >>
\Omega_{\PP(E_r)/k} @>>> 0. 
\end{CD}
$$
Using the Euler exact sequence, $\nu^*\mc{J}$ is identified with the
locally free sheaf of rank $n-r$, that is the kernel of the associated
surjective morphism
$$
\wt{d\nu}^\dagger: E^\vee\otimes_k \OO_{E_r}(-e)\to E_r^\vee\otimes_k \OO_{E_r}(-1).
$$
Via the factorization of $\nu$ through $\PP_k(E_\nu)$, there is an
associated short exact sequence for $\nu^*\mc{J}$,
$$
0\to J_1\otimes_k \OO_{E_r}(-e) \to \nu^*\mc{J} \to \nu^*\mc{J}_{>1} \to 0,
$$
where $\mc{J}_{>1}$ is the ideal sheaf of $\text{Image}(\nu)$ in
$\PP_k(E_\nu)$. 

\mni 
Denote by
$\Delta_r(e)\subset \ZZ_{\geq 0}^{r+1}$  the subset of
$\underline{e}=(e_0,e_1,\dots,e_r)$ with $e_0+e_1+\dots+e_r$ equal to
$e$.  This set has size $\pr(e)$.  Denote by $m$ the difference $n+1-\pr(e)$.
Denoting by $(t_0,\dots,t_r)$ a basis for $E_r^\vee$, and denoting by
$(y_1,\dots,y_m)$ a basis for $J_1$, this extends to  
a basis for $E^\vee$,
$$
(x_{\underline{e}})_{\underline{e} \in \Delta_r(e)}\sqcup
(y_1,\dots,y_m),
$$
such that for every $\underline{e}=(e_0,e_1,\dots,e_r)$,
$$\nu_1^*x_{\underline{e}} = t_0^{e_0}t_1^{e_1}\cdots
t_r^{e_r}.
$$ 
Then the restriction
$k$-algebra homomorphism $\nu^*:
S(E)\twoheadrightarrow S(E_r)$ is the composition of the
quotient $k[x_{\underline{e}},y_j]\to k[x_{\underline{e}}]$ by the graded ideal
generated by $(y_1,\dots,y_m)$ and the natural surjection
$k[x_{\underline{e}}] \to k[t_0,\dots,t_r]_{(e)}$, where
$k[t_0,\dots,t_r]_{(e)}$ is the graded $k$-subalgebra $\oplus_{d\geq
  0} k[t_0,\dots,t_r]_{de}$ of $k[t_0,\dots,t_r]$.  
The linear space $\text{Zero}(y_1,\dots,y_m)$ equals the linear span of
the image of $\nu$, $\text{Span}(\nu)$.
  
\mni
The identity map $S_1(E)\to E^\vee$ extends uniquely to a
$k$-derivation that also preserves graded decompositions,
$$
\partial : S(E)\to E^\vee\otimes_k S(E)[-1].
$$
This $k$-derivation defines a graded isomorphism of $S(E)$-modules,
$$
\Omega_{S(E)/k} \to E^\vee\otimes_k S(E)[-1].
$$ 
Similarly, the identity map $S_1(E_r)\to E_r^\vee$ defines a graded
isomorphism of $S(E_r)$-modules,
$$
\Omega_{S(E_r)/k} \to E_r^\vee\otimes_k S(E_r)[-1].
$$
In particular, if $e$ is prime to the characteristic,
then the derivation in degree $e$,
$$
S_e(E_r) \to E_r^\vee\otimes_k S_{e-1}(E),
$$
defines a surjection of $\OO_{\PP(E_r)}$-modules,
$$
\partial_e:S_e(E_r)\otimes_k\OO_{\PP(E_r)} \to E_r^\vee\otimes_k\OO_{E_r}(e-1).
$$
Twisting, this gives an isomorphism,
$$
\nu^*\mc{J}_{>1}(e) \cong \text{Ker}(\partial_e).
$$
If the characteristic does divide $e$, then $\partial_e$ has cokernel
isomorphic to $\OO_{\PP(E_r)}(e)$, and then there is a short exact
sequence,
$$
0 \to \nu^*\mc{J}_{>1}(e) \to \text{Ker}(\partial_e) \to \OO_{E_r}(e)
\to 0.
$$

\begin{lem} \label{lem-H1J} \marpar{lem-H1J}
Each of $\nu^*\mc{J}_1$, $\nu^*\mc{J}$, and $\nu^*\mc{J}_{>1}$ is a
locally free $\OO_{\PP(E_r)}$-module.  Moreover, for each, $h^1$ of
the dual locally free sheaf is zero.
\end{lem}

\begin{proof}
Via the identifications above, it is straightforward to compute that
each sheaf is locally free.  Moreover, since $\nu^*\mc{J}_1^\vee$ is
isomorphic to $J_1^\vee\otimes_k \OO_{E_r}(e)$, all of the higher
cohomology groups of this sheaf are zero.  Via the long exact sequence
of cohomology, $h^1$ of $\nu^*\mc{J}^\vee$ equals zero if $h^1$ of
$\nu^*\mc{J}_{>1}^\vee$ equals zero.  Via the isomorphisms and via the
vanishing of higher cohomology of $\OO_{\PP(E_r)}$, $h^1$ of
$\nu^*\mc{J}_{>1}^\vee$ equals zero if $h^1$ if
$\text{Ker}(\partial_e)^\vee(e)$ equals zero.  If the characteristic
is prime to $e$, resp. divides $e$, then we have an exact
sequence,
$$
0 \to E_r\otimes_k \OO_{E_r}(1) \to S_e(E_r)^\vee\otimes_k
\OO_{E_r}(e) \to \text{Ker}(\partial_e)^\vee(e) \to 0,
$$
resp. we have an exact sequence,
$$
0 \to \OO_{\PP(E_r)}\to E_r\otimes_k \OO_{E_r}(1) \to S_e(E_r)^\vee\otimes_k
\OO_{E_r}(e) \to \text{Ker}(\partial_e)^\vee(e) \to 0.
$$
In each case, using the vanishing of all higher cohomology of
$\OO_{E_r}(\ell)$ for $\ell > -r$, the long exact sequence of
cohomology implies that $h^1$ of $\text{Ker}(\partial_e)^\vee(e)$
equals $0$.  
\end{proof}

\mni
For every $G\in S_d(E)$, i.e., for every global section of
$\OO_E(d)$, the associated adjoint map,
$$
\partial G^\dagger : E\to S_{d-1}(E),
$$ 
sends each dual basis vector of $E$, 
$x_{\underline{e}}^\vee$, resp. $y_i^\vee$ to the partial
derivative $\partial G/\partial x_{\underline{e}}$, resp. $\partial
G/\partial y_i$.    
Assume now that $G$ vanishes on $\text{Span}(\nu)$.  (N.B.  When
$e=1$, this hypothesis is trivially satisfied.  
For $e>1$, one could try to
improve the bound in the theorem by dropping this hypothesis.)
Denote by $i:Y\hookrightarrow
\PP_k(E)$ the zero scheme of $G$, and denote by $\mc{I}_Y$ the
corresponding ideal sheaf of $\OO_{\PP(E)}$.  
Multiplication by $G$ defines an
isomorphism of $\OO_{\PP(E)}$-modules, $\OO_{E}(-d)\to \mc{I}_Y$.
Thus the fundamental exact sequence of sheaves of relative
differentials becomes,
$$
\begin{CD}
0 @>>> i^*\OO_E(-d) @> \partial G >> i^*\Omega_{\PP(E)/k} @> (di)^\dagger >>
\Omega_{Y/k} @>>> 0. 
\end{CD}
$$
Pulling back to $\PP_k(E_\nu)$ 
and using transitivity for relative
differentials, there is a commutative diagram, 
$$
\begin{CD}
\OO_{E_\nu}(-d) @> \nu^*\partial G >> \Omega_{\PP(E)/k}|_{\PP(E_\nu)} \\
@V \partial G_{\nu} VV @VVV \\
\mc{J}_1 @> \delta >> \Omega_{\PP(E_r)/k}
\end{CD}
$$
Via the identification of $\mc{J}_1$, the homomorphism $\partial
G_{\nu}$ 
is
equivalent to a homomorphism,
$$
\OO_{E_\nu}(-d) \to (E/E_\nu)^\vee\otimes_k \OO_{E_\nu}(-1),
$$
Up to a twist and taking the transpose, this is equivalent to a homomorphism,
$$
\partial G_{\nu,1}^\dagger: (E/E_\nu)\otimes_k \OO_{\PP(E_\nu)} \to \OO_{E_\nu}(d-1).
$$ 
Via adjointness of pushforward and pullback, this is equivalent to a
homomorphism of $k$-vector spaces,
$$
E/E_\nu \to S_{d-1}(E_\nu).
$$
By abuse of notation, this is also denoted by $\partial G_\nu^\dagger$.
This map fits into a commutative diagram,
$$
\begin{CD}
E @> \partial G^\dagger >> S_{d-1}(E) \\
@VVV @VVV \\
E/E_\nu @>> \partial G_{\nu,1}^\dagger > S_{d-1}(E_\nu)
\end{CD},
$$
where the vertical arrows are the natural surjections.
Composing with the surjection $\nu_e^*$, 
this map induces a $k$-linear
map,
$$
G_\nu^\dagger:E/E_\nu \to S_{(d-1)e}(E_r)
$$

\mni
For every integer $c\geq 0$, for every integer $b\geq 0$, for every
$k$-vector space $W$ and $k$-linear map $\phi:W\to
S_{b}(E)$, there is an associated $k$-linear map
$$
\phi_c:W\otimes_k S_c(E) \to S_{b+c}(E).
$$
obtained from the multiplication on $S(E)$.  The $k$-linear map $\phi$
is $c$-\emph{generating} if $\phi_c$ is surjective,
cf. \cite[Definition 7.2]{HS2}.

\begin{lem} \label{lem-normal} \marpar{lem-normal}
Under the above hypothesis that the degree $d$ hypersurface $Y$
contains the linear span $\PP(E_\nu)$ of $\nu$, 
the smooth locus $Y^o$ of the $k$-scheme $Y$ contains
$\text{Image}(\nu)$, 
resp.
$\PP(E_\nu)$, if and only if 
the linear system $\partial G_\nu^\dagger$ on $\PP(E_r)$,
resp. the linear system $\partial G_{\nu,1}^\dagger$ on $\PP(E_\nu)$, 
is $c$-generating for some $c\geq 1$.  When $G^o$ contains $\text{Image}(\nu)$,
$p$ is smooth at the point corresponding to the pair
$([\text{Image}(\nu)],[Y])$ if and only if $\partial G_\nu^\dagger$ 
is $e$-generating.
\end{lem}

\begin{proof}
By the Jacobian criterion, $Y^o$ equals the maximal open subscheme of
$Y$ on which $\Omega_{Y/k}$ is locally free of rank $n-1$.  By
Nakayama's Lemma, the points of $Y$ are precisely those $\SP \kappa
\to Y$ where the pullback of $\partial G^\dagger$ is nonzero.  By the
commutative diagram, this is equivalent to nonvanishing of the
pullback of $\partial G_j^\dagger$.  Finally, using the equivalence
between the category of coherent sheaves on $\PP_k(E_r)$ and the
localization of the
category of finitely presented, graded $S(E_r)$-modules with respect
to modules concentrated in low degrees, 
this pullback is nonvanishing at every point of $\PP_k(E_r)$,
resp. $\PP_k(E_\nu)$, 
if and only if $\partial G_\nu^\dagger$, resp. $\partial
G_{\nu,1}^\dagger$ 
is $c$-generating for some
$c\geq 1$.

\mni
Next, assume that $\partial G_c^\dagger$ is $c$-generating for some
$c\geq 1$.  Then for the ideal sheaf $\mc{K}$ of $\PP(E_r)$ in $Y$,
the commutative diagram gives a short exact sequence,
$$
\begin{CD}
0 @>>> \OO_{E_r}(-de) @>\partial G_j >> \nu^*\mc{J} @>>>
\nu^*\mc{K} @>>> 0 
\end{CD}
$$
Since $Y^o$ contains $\text{Image}(\nu)$, the closed immersion
$\nu\PP(E_r)\hookrightarrow Y^o$ is a \emph{regular immersion}.  Thus
the usual obstruction group for deformations of this closed immersion,
$\text{Ext}^1_{\OO_Y}(\mc{K},\nu_*\OO_{\PP(E_r)})$ reduces to
$$
H^1( \PP(E_r), \textit{Hom}_{\OO_{\PP(E_r)}}(\nu^*\mc{K},\OO_{\PP(E_r)})).
$$  
Since
$\nu^*\mc{K}$ is locally free, the transpose of the short exact sequence
above is still a short exact sequence.  By Lemma \ref{lem-H1J}, 
the long exact sequence defines
an isomorphism
$$
\delta:\text{Coker}(\partial G_j^\dagger)_1 \xrightarrow{\cong} 
H^1( \PP(E_r), \textit{Hom}_{\OO_{\PP(E_r)}}(j^*\mc{K},\OO_{\PP(E_r)})).
$$
Thus, the obstruction group vanishes if and only if $\partial
G_\nu^\dagger$ is $e$-generating.  

\mni
Of course there are cases where the
obstruction group is nonzero, yet the relative 
Hilbert scheme is still smooth.
However, in this case, both the domain and the target of the morphism
$p$ are smooth $k$-schemes.  
The obstruction group is the cokernel of the map induced by
$p$ from the Zariski
tangent space of $\Fe(r,\mc{X}/\PP_k(S_d(E)))$ to the Zariski tangent
space of $\PP_k(S_d(E))$.  Thus, by the Jacobian criterion,
$p$ is smooth at $([\text{Image}(\nu)],[Y])$ 
if and
only if $\partial G_\nu^\dagger$ is $e$-generating.
\end{proof}

\mni
By Hochster-Laksov \cite[Theorem 1]{HochsterLaksov}, for $e=1$, for
all $d\geq 3$, for all 
$n\geq \N{1}(r,d) = r+\lceil \pr(d)/(r+1) \rceil$, there exists a linear
system of dimension $m=n-r$ in $S_{d-1}(E_r)$ that is $1$-generating,
say
$$
E/E_\nu \to S_{d-1}(E_r), \ \ y_i^\vee \mapsto G_i(t_0,\dots,t_r).
$$
Recall that the basis for $E^\vee$ is $(x_0,\dots,x_r)\sqcup
(y_1,\dots,y_m)$, where $\nu^*x_i$ equals $t_i$.  
For the polynomial
$$
G = \sum_{i=1}^m y_i G_i(x_0,\dots,x_r),
$$
the zero scheme, $Y$, of $G$ 
contains $\PP(E_r)$, and $\partial G_\nu^\dagger$ is the given
$1$-generating linear system.  Thus, by Lemma \ref{lem-normal}, $\pi$ is
smooth at the pair $([\PP(E_r)],[Y])$.  This proves Theorem
\ref{thm-Waldron}, and this is basically Waldron's proof.  In fact,
Waldron also gives a simplified proof of Hochster-Laksov in this case.

\mni
Next, for $e\geq 2$, for all $d\geq 3$, it is a theorem of Gleb
Nenashev, \cite[Theorem 1]{Nenashev}, 
that for all integers $m=n-\pr(e)$ satisfying
$m\geq \pr(e) + \lceil \pr(de)/\pr(e) \rceil$, there exists an
$e$-generating linear system,
$$
E/E_\nu \to S_{(d-1)e}(E_r), \ \ y_i^\vee \mapsto G_i(t_0,\dots,t_r).
$$
For each $i$, since $\nu^*_{d-1}$ is surjective, there exists $H_i\in
k[x_{\underline{e}}]_{d-1}$ such that $\nu^*{d-1}(H_i)$ equals $G_i$.
For the polynomial
$$
G = \sum_{i=1}^m y_i H_i,
$$
the zero scheme, $Y$, of $G$ contains $\PP(E_\nu)$, and $\partial
G_\nu^\dagger$ is the given $e$-generating linear system.  Thus, by
Lemma \ref{lem-normal}, $\pi$ is smooth at the pair $([\PP(E_r)],[Y])$.  
This proves the Theorem \ref{thm-Veronese} for $e\geq 2$ and for
$d\geq 3$.

\mni
For $d=1$, for all $n\geq 1+n_r(e)$, for every hypersurface $Y$ that
contains $\text{Image}(\nu)$, $\Fe(r,Y) \cong \Gge(r,\PP^{n-1}_k)$ is
nonempty and smooth.  For $d=2$ and for $e\geq 2$, there are smooth
quadric surfaces that contain $\text{Image}(\nu)$, assuming that $k$
is algebraically closed (it would suffice for $k$ to be infinite).  
This follows most easily from Bertini's theorem.  Since $r\geq 1$,
also $2r+1 \geq 2$.  Thus, $\pr(2) \geq 2r+2$.  By Pascal's Theorem,
$\pr(t+1) - \pr(t)$ equals $P_{r-1}(t+1)$.  For $e\geq -r$, resp. for
$e\geq -1$,
$P_{r-1}(e+1)\geq 0$, resp. $P_{r-1}(e+1)>0$, 
so that the integer-valued function
$\pr(e)$ is nondecreasing, resp. increasing, in $e$ for $e\geq -r$,
resp. $e\geq -1$.  Thus, for all $e\geq 2$, $\pr(e)\geq \pr(2) \geq
2r+2$.  Thus, for $n\geq n_r(e) = \pr(e)-1$, $n$ is strictly larger
than $2r$, $\pr(e)-1 \geq 1+2r$.  Thus, by the usual parameter
counting 
proof of
Bertini's theorem, to prove that a general member $G$ in
$H^0(\PP_k(E),\mc{J}(2))$ is defines an everywhere smooth quadric, it
suffices to prove for every $k$-point $p\in \PP_k(E_r)$ that the
induced map,
$$
H^0(\PP_k(E),\mc{J}(2)) \to T_{\nu(p)} \PP_k(E) / d\nu( T_p \PP_k(E_r)), 
$$
is surjective.  

\mni
Choose homogeneous coordinates on $\PP_k(E_r)$ so
that $p$ equals $[t_0,t_1,\dots,t_r] = [1,0,\dots,0]$, and then choose
corresponding homogeneous coordinates $(x_{\underline{e}},y_i)$ on
$\PP_k(E)$ as above.
Then $\nu(p)$
is the point where the coordinate $x_{(e,0,\dots,0)}\neq 0$, yet
$x_{\underline{e}}=0$ for every $\underline{e} \in
\Delta_r(e)\setminus\{ (e,0,\dots,0)\}$, 
and $y_i=0$ for every $i=1,\dots,m$.  

\mni
The tangent space of $T_p\PP_k(E_r)$ is the space
spanned by the partial derivatives
$\partial/\partial(x_{\underline{e}}/x_{(e,0,\dots,0)})$ for the
elements $\underline{e} = (e-1,0,\dots,0,1,0,\dots,0)$.  The quotient
space is generated by the partial derivatives for
$y_i/x_{(e,0,\dots,0)}$ for $i=1,\dots,m$, and by the partial
derivatives of $x_{\underline{e}}/x_{(e,0,\dots,0)}$, where
$\underline{e}=(e_0,e_1,\dots,e_r)$ satisfies $e_0 \leq e-2$.  
For every
$i=1,\dots,m$, the quadratic polynomial $y_i x_{(e,0,\dots,0)}$ maps
to the image of the partial derivative for $y_i/x_{(e,0,\dots,0)}$.
For every $\underline{e}\in \Delta_r(e)$ with $e_0\leq e-2$, there
exist elements 
$\underline{e}', \underline{e}''\in \Delta_r(e)$ with $e_0',e''_0\leq
e-1$ such that $\underline{e}+(e,0,\dots,0) =
\underline{e}'+\underline{e}''$.  Thus the quadratic polynomial
$x_{\underline{e}}x_{(e,0,\dots,0)} -
x_{\underline{e}'}x_{\underline{e}''}$ maps to the image of the
partial derivative for $x_{\underline{e}}/x_{(e,0,\dots,0)}$.  Thus,
by Bertini's Theorem, there exists $G\in H^0(\PP_k(E),\mc{J}(2))$ such
that $Y=\text{Zero}(G)$ is everywhere smooth.  The action of
$\textbf{PGL}(E)$ on the open subset $\PP S_2(E)\setminus \Delta$
parameterizing smooth quadrics is smoothly homogeneous.  This action
lifts to an action of $\textbf{PGL}(E)$ on
$\Fe(r,\mc{X})$.  Thus, whenever $Y$ is smooth, the restriction of $\pi$
to the $\textbf{PGL}(E)$-orbit of $([\text{Image}(\nu)],[Y])$ is
smooth, cf. \cite[Corollaire 6.5.2(i)]{EGA4}.

\mni
By the lemma, if $\pi$ is smooth, then $n-r \geq n_0-r$.  Conversely,
assume that $n \geq n_0$.  Then there exists a $k$-linear map,
$$
\phi:E/E_r \to S_{d-1}(E_r)
$$
that is $1$-generating.  For the image of every dual vector $t_i^\vee$
in $E/E_r$, denote by $G_i\in S_{d-1}(E_r)$ the image of this element
under $\phi$.  For every $(n-r)$-tuple
$(\widetilde{G}_{r+1},\dots,\widetilde{G}_n)$ of elements 
$\widetilde{G}_i\in S_{d-1}(E)$ 
that maps to $G_i$, 
the element
$$
G = \widehat{G} + t_{m+1}\widetilde{G}_1 + \dots + t_n\widetilde{G}_n
$$
is an element of $S_d(E)$ that vanishes on $\PP(E_r)$.  By definition
of $\partial G_j^\dagger$ in terms of partial derivatives, this equals
$\phi$.  Thus, $([\PP(E_r)],[Y])$ is a point where $p$ is smooth.
This proves the first part of the proposition: the smooth locus of $p$
is nonempty (and hence dense) if and only if $n\geq n_0$.  


\section{Proof of Proposition \ref{prop-dim}} \label{sec-dim}
\marpar{sec-dim}

\mni
In this section, fix $e$ to equal $1$.
Let $n$ equal $n_1'(d,r)$, and
denote $m=n_1'(d,r)-r$.  The morphism $\rho:\Fe(r,\mc{X}) \to
\Gge(r,\PP(E))$ is a Zariski locally trivial projective bundle.
Moreover, both domain and target have natural actions of
$\textbf{PGL}(E)$, and the morphism is equivariant for these actions.
Finally, the morphism $\pi$ is also equivariant.  Thus the closed
subscheme $B$ where $\pi$ is not smooth is
$\textbf{PGL}(E)$-invariant.  Since $\Gge(r,\PP(E))$ is homogeneous
under the action of $\textbf{PGL}(E)$, the restriction of $\rho$ to
$B$ is flat.  
Thus, the hypothesis that $B$ has codimension $\geq 2$
everywhere is equivalent to the hypothesis that the intersection of
$B$ with one, and hence every, geometric fiber of $\rho$ has
codimension $\geq 2$ everywhere in that fiber.  Since the geometric
fibers of $\rho$ are projective spaces, this is equivalent to the
hypothesis that 
there exists a finite
morphism from $\PP^1$ to a geometric fiber of $\rho$ whose image is
disjoint from $B$.

\mni
Now we apply the proof of Bertini's Connectedness Theorem, as
generalized by Cristian Minoccheri.  For the $k$-morphism,
$$
\pi:\F{1}(r,\mc{X}) \to \PP_k(S_d(E)),
$$
the source and target are both smooth, projective $k$-schemes, and the
target is algebraically simply connected, since it is a projective
space.  By hypothesis, the complement $B$ of the smooth locus of $\pi$ has
codimension $\geq 2$ everywhere.  
Thus, by \cite[Theorem 3.1]{Minoccheri}, 
the geometric generic fiber of $\pi$ is connected.
Finally, by Zariski's Main Theorem, since the geometric generic fiber
of $\pi$ is connected, every geometric fiber of $\pi$ is connected.

\mni
For the maximal open subscheme $U$ of $\PP_k(S_d(E))$ over which both
$\F{1}(r,\mc{X})$ and $B$ are flat, the restriction of 
$\pi$ over $U$ is a flat morphism whose domain and target are both
smooth, hence $\pi|_U$ is a flat, local complete intersection morphism
\cite[Appendix B.7.6]{F}.  So every geometric fiber of $\pi$ over $U$ is a
projective, local complete intersection scheme.  Moreover, the
singular locus equals the intersection of the fiber with $B$, and this
has codimension $\geq 2$ by hypothesis.  Thus, by Serre's Criterion,
the geometric fiber is integral and normal, cf. \cite[Th\'{e}or\`{e}me
5.8.6]{EGA4}.  As in
Proposition \ref{prop-inc}, if the characteristic equals $0$ or $p\geq
p_e(n,r,d)$, then, up to replacing $U$ by a dense, Zariski open
subscheme, $\pi|_U$ is even smooth.


\section{Proof of Corollary \ref{cor-dim}} \label{sec-dim2}
\marpar{sec-dim2}

\mni
Recall from Corollary \ref{cor-VeroneseV} that there exists a
universal family of linear spans $\PP_G(E_G) \subset
\Gge(r,\PP_k(E))\times_{\SP(k)} \PP_k(E)$ of the universal family of
Veronese varieties.  This projective subbundle contains the universal
family of Veronese $e$-uple $r$-folds.  The pair defines a morphism to
the flag Hilbert scheme,
$$
\Psi_{e,r,E}:\Gge(r,\PP_k(E))\to
\text{fHilb}^{\pr(et),P_{\nee(r)}(t)}_{\PP(E)/k},
$$
whose image is contained in the inverse image open subset
$\Phi^{-1}(\Gge(r,\PP_k(E)))$, where $\Phi$ is the forgetful morphism
$$
\Phi:
\text{fHilb}^{\pr(et),P_{\nee(r)}(t)}_{\PP(E)/k}
\to
\text{Hilb}^{\pr(et)}_{\PP(E)/k}.
$$
Now consider the second forgetful morphism,
$$
\Lambda:\Phi^{-1}(\Gge(r,\PP_k(E))) \to \Gg{1}(\nee(r),\PP_k(E)).
$$
Using the action of $\textbf{PGL}(E)$, the morphism $\Lambda$ is a
Zariski locally trivial fiber bundle whose fiber over a $\kappa$-valued
point 
$[\PP_\kappa(E')]\in
\Gg{1}(\nee(r),\PP_k(E))(\SP(\kappa))$ equals $\Gge(r,\PP_\kappa(E'))$.  In
particular, $\Lambda$ is faithfully flat, finitely presented,
quasi-projective and smooth with geometrically irreducible fibers.  

\mni
Now let $d\geq 1$ be an integer.  As usual, denote by $\mc{X}\subset
\PP H^0(\PP_k(E),\OO_E(d))\times_{\SP(k)} \PP_k(E)$ the universal
family of degree $d$ hypersurfaces in $\PP_k(E)$.  
Consider the projection
$$
\rho:\F{1}(\nee(r),\mc{X}) \to \Gg{1}(\nee(r),\PP_k(E)).
$$
Denote by $F_{e,1}(r,\nee(r),\mc{X})$ the fiber product,
$$
\begin{CD}
F_{e,1}(r,\nee(r),\mc{X}) @>\text{pr}_1 >> \Phi^{-1}(\Gge(r,\PP_k(E))) \\
@V \text{pr}_2 VV @VV \Lambda V \\
\F{1}(\nee(r),\mc{X}) @>> \rho > \Gg{1}(\nee(r),\PP_k(E)).
\end{CD}
$$
Chasing diagrams, $F_{e,1}(r,\nee(r),\mc{X})$ is an open subset of the
relative flag Hilbert scheme of $\pi:\mc{X}\to \PP
H^0(\PP_k(E),\OO_E(d))$ parameterizing pairs of closed subschemes in
fibers of $\pi$ of Hilbert polynomials $\pr(et)$,
resp. $P_{\nee(r)}(t)$.  More precisely, this is the open subset of the
flag Hilbert scheme parameterizing pairs where the smaller closed
subscheme is a Veronese $e$-uple $r$-fold, and where the larger closed
subscheme is the linear span of the Veronese variety.  In particular,
because $\Lambda$ is faithfully flat, finitely presented,
quasi-projective and smooth with geometrically irreducible fibers, the
same holds for the base change morphism,
$$
\text{pr}_2 : F_{e,1}(r,\nee(r),\mc{X}) \to \F{1}(\nee(r),\mc{X}).
$$

\mni
Since $n\geq n'_1(d,\nee(r))$, the projection morphism
$$
\pi':\F{1}(\nee(r),\mc{X})\to \PP H^0(\PP_k(E),\OO_E(d))
$$
is projective and dominant with irreducible geometric generic fiber,
by Proposition \ref{prop-dim}.  Combined with the previous paragraph,
also the composition
$$
F_{e,1}(r,\nee(r),\mc{X}) \xrightarrow{\text{pr}_2} \F{1}(\nee(r),\mc{X})
\xrightarrow{\pi'} \PP H^0(\PP_k(E),\OO_E(d)),
$$
is quasi-projective and dominant with irreducible geometric generic
fiber.  By the definition of the flag Hilbert scheme, this composition
also equals the composition
$$
F_{e,1}(r,\nee(r),\mc{X})\xrightarrow{\text{pr}_1} \Fe(r,\mc{X})
\xrightarrow{\pi} \PP H^0(\PP_k(E),\OO_E(d)).
$$
By Proposition \ref{prop-inc}, $\Fe(r,\mc{X})$ is smooth and
irreducible, even a projective bundle over $\Gge(r,\PP_k(E))$.  In
particular, the image of $\text{pr}_1$ is contained in the normal
locus of $\Fe(r,\mc{X})$.  Thus, by \cite[Lemma 3.2]{dJS10},
also the morphism
$$
\pi:\Fe(r,\mc{X})\to \PP H^0(\PP_k(E),\OO_E(d))
$$
is dominant with irreducible geometric generic fiber.  By the usual
constructibility argument, cf. \cite[Th\'{e}or\`{e}me I.4.10]{Jou},
there exists a
dense open subset $W^r_{e,d}$ of $\PP H^0(\PP_k(E),\OO_E(d))$ over which $\pi$ is
faithfully flat with geometrically irreducible fibers.


\section{Proof of Theorem \ref{thm-nprime}} \label{sec-nprime}
\marpar{sec-nprime}

\mni
\textbf{Proof of (i).} Since $d\geq 2$, also $d-1\geq 1$.  Thus
the difference $m=n-r-1$ satisfies
$m\geq n_1(d,r)-r$, and this is at least $r+1$.  
In particular, $m$ is
positive.  

\mni
Choose
homogeneous coordinates $(x_0,\dots,x_r,y_0,y_1,\dots,y_m)$ on
$\PP_k(E)$ so that $\PP_k(E_r)$ equals the zero scheme of
$(y_0,y_1,\dots,y_m)$.  
Assume that $n\geq 1+n_1(r)$.  Then by Hochster-Laksov once again, 
there exists a $1$-generating
$k$-subspace 
$\iota:W\hookrightarrow 
k[t_0,\dots,t_r]_{d-1}$ 
of dimension $m$.  Let $(w_1,\dots,w_m)$ be an ordered basis
for $W$, and denote by $G_i(t_0,\dots,t_r)$ the image $\iota(w_i)$.  
Define $W'=k^{\oplus (m+1)}$, and define
$$
\psi:(k^{\oplus 2}) \times (k^{\oplus(m+1)}) \to W, 
$$
$$
\psi((a,b),(c_0,\dots,c_m)) = (ac_0+bc_1)w_1 + (ac_1+bc_2)w_2 +
\dots + (ac_{m-1}+bc_m)w_m.
$$
This is a $k$-bilinear map.
When $a$ is nonzero, then the restriction of $\psi_{(a,b),\bullet}$ to the
subspace $\text{Zero}(c_m)$ is an isomorphism, so that the image is
$1$-generating.  When $b$ is nonzero, then the restriction of
$\psi_{(a,b),\bullet}$ to the subspace $\text{Zero}(c_0)$ is an
isomorphism.  Thus, defining
$$
G_{a,b}(x_i,y_j) =
(ay_0+by_1)\Gg{1}(x_0,\dots,x_r) + \dots + (ay_{m-1}+by_m)G_m(x_0,\dots,x_r),
$$
and defining $Y_{a,b} = \text{Zero}(G_{a,b}) \subset \PP_k(E)$,
there is a morphism
$$
g:\PP^1_k \to \F{1}(r,\mc{X}), \ \ [a,b] \mapsto ([\PP_k(E_r)],[Y_{a,b}]),
$$
that is a finite morphism into the fiber of $\rho$ over $[\PP_k(E_r)]$
whose image is contained in the smooth locus
$\Fe(r,\mc{X})_{\text{sm}}$ of $\pi$, i.e., the image of $g$ is
disjoint from the singular locus $B$ of $\pi$.  The fiber of $\rho$
is a projective space, and every nonempty Cartier divisor in
projective space has nonempty intersection with every nonempty curve
in projective space.  Thus, the intersection of $B$ with the fiber of
$\rho$ has codimension $\geq 2$ in that fiber.
As in the proof of Proposition
\ref{prop-dim}, since both $\rho$ and $\rho|_B:B\to \Gg{1}(r,\PP_k(E))$
are flat, it follows that $B$ has codimension $\geq 2$ everywhere in
$\F{1}(r,\mc{X})$.  
Thus,
$n_1'(d,r)$ is no greater than $1+n_0(d,r)$.

\mni
\textbf{Proof of (ii).}  This is essentially the same argument as in
the proof of Corollary \ref{cor-dim}.  The morphism $\Lambda:\Fe(r,\mc{X})
\to \PP H^0(\PP_k(E),\OO_E(1))$ is a Zariski locally trivial fiber
bundle whose fibers are schemes $\Gge(r,\PP^{n-1}_k)$.  These are
nonempty and geometrically connected precisely when $n-1\geq \nee(r)$,
i.e., $n\geq 1 + \nee(r)$.  

\mni
\textbf{Proof of (iii).}  By (i), $n'_1(2,r)\leq 1+n_1(2,r)$.  Thus,
it suffices to prove that $\F{1}(r,X)$ is disconnected for
$n=n_1(2,r)=2r+1$.   Denote by
$(t_0,\dots,t_r,t_{r+1},\dots, t_{2r+1})$ an ordered basis for
$S_1(E)$.  
For $k$ algebraically
closed, all smooth quadric hypersurfaces in $\PP_k(E)$ are
projectively equivalent to the zero scheme of the quadratic
polynomial,
$$
G(t_0,t_1,\dots,t_r,t_{r+1},t_{r+2},\dots,t_{2r+1}) = t_0t_{r+1} +
t_1t_{r+2} + \dots + t_rt_{2r+1}.
$$
Consider the $r$-planes
$\Pi=\text{Zero}(t_{r+1},\dots,t_{2r+1})$,
$\Lambda=\text{Zero}(t_0,\dots,t_r)$, and $\Gamma =
\text{Zero}(t_0,t_{r+2},t_{r+3},\dots,t_{2r+1})$.  These are all
contained in $X=\text{Zero}(G)$.  By \cite[Lemma 0.3, Appendix]{BrH-B},
$(\Pi.\Lambda)_X = 0$.  If $m$ is even, then
$(\Pi.\Pi)_X = 1$.  If $m$ is odd, then $(\Pi.\Gamma)_X=1$.  Since
algebraically equivalent $m$-cycles are numerically equivalent, it
follows that $F_m(X/k)$ has more than one connected component (in fact
it has precisely two connected components).

\mni
\textbf{Proof of (iv).}  This is a computation in the Chow group of
the Grassmannian $\Gg{1}(r,\PP_k(E))$.  As an Abelian group under
addition, this is a finite free Abelian group.  Moreover, for the
projection of the flag variety to the Grassmannian,
$$
\text{Flag}(0,1,\dots,r,\PP_k(E)) \to \Gg{1}(r,\PP_k(E)),
$$
the induced pullback map on Chow rings is an injective homomorphism
that identifies the Chow ring of $\Gg{1}(r,\PP_k(E))$ with a saturated
Abelian subgroup of the Chow ring of the flag variety, i.e., the
quotient Abelian group is a finite free Abelian group.  Thus,
divisibility of cycles in the Chow ring of $\Gg{1}(r,\PP_k(E))$ can be
checked after pullback to the Chow ring of the flag variety.

\mni
There are many methods for
performing computations in the Chow ring of the flag variety.  
The method used here is via ``Chern roots'' of the total Chern class
of the tautological bundle.
Begin with the $(r+1)$-fold fiber product,
$$
\PP_k(E)^{r+1} = \PP_k(E) \times_{SP(k)} \dots \times_{\SP(k)} \PP_k(E).
$$
Inside of this scheme, denote by $D_{\leq r}$ the degeneracy  closed
subscheme (in the sense of Porteous's formula)
where
the $r+1$ points are linearly degenerate, i.e., the closed subscheme
defined by the vanishing of all $(r+1)\times(r+1)$-minors 
of the following homomorphism of
locally free sheaves,
$$
\phi_{r+1}:E^\vee\otimes_k \OO_{\PP(E)^{r+1}} \to \bigoplus_{i=0}^r
\text{pr}_i^* \OO_{\PP(E)}(1).
$$
Denote by $U_{r+1}\subset \PP_k(E)^{r+1}$ the open complement of
$D_{\leq r}$.  On $U_{r+1}$, the morphism $\phi_{r+1}$ is a locally
free quotient of rank $r+1$.  Thus, for every integer $0\leq s\leq r$, the
associated map
$$
\phi_{r+1,s+1}:E^\vee\otimes_k \OO_{U_{r+1}} \to \bigoplus_{i=0}^s
\text{pr}_i^*\OO_{\PP(E)}(1) 
$$
is a locally free quotient of rank $s+1$.  Altogether, these morphisms
define a morphism to the flag variety,
$$
\beta:U_{r+1}\to \text{Flag}(0,1,\dots,r,\PP_k(E)).
$$
Working inductively on $r$, $\beta$ is an iterated fiber bundle, each
factor of which is an affine space bundle that is trivialized for a Zariski
open covering of the target.  Thus, by the homotopy axiom for Chow
groups, \cite[Proposition 1.9]{F}, the pullback map
$$
\text{CH}^*(\text{Flag}(0,1,\dots,r,\PP_k(E))) \to \text{CH}^*(U_{r+1})
$$
is a ring isomorphism.  

\mni
On the other hand, since $U_{r+1}$ is an open subset of
$\PP_k(E)^{r+1}$, there is a presentation for the Chow group,
$$
\text{CH}^*(\PP_k(E)^{r+1})/I \cong \text{CH}^*(U_{r+1}),
$$
where $I$ is a $\mathfrak{S}_{r+1}$-invariant ideal.  Since
$\text{CH}^*(\PP_k(E)) = \ZZ[u]/\langle u^{n+1} \rangle$ for the first
Chern class $u$ of $\OO_E(1)$, this presentation is the same as
as
$$
\ZZ[u_0,\dots,u_r]/J
$$
where $J$ is an ideal containing $\langle
u_0^{n+1},\dots,u_r^{n+1}\rangle$ such that $J/\langle
u_0^{n+1},\dots,u_r^{n+1} \rangle$ equals $I$, and each $u_i$ is the
first Chern class of $\text{pr}_i^*\OO_{\PP(E)}(1)$.  Moreover, for
every integer $s=0,\dots,r$, for the tautological locally free
quotient bundle $E^\vee\otimes_k \OO_{\text{Flag}}\to Q_{s+1}$ of rank
$s+1$, 
the total
Chern class of $Q_{s+1}$ equals the image of
$(1+u_0)(1+u_1)\cdots(1+u_s)$.
Thus, the elements $(u_0,\dots,u_r)$ are the ``Chern roots'' of the
tautological flag of locally free sheaves on the flag variety (up to
signs, depending on the sign convention; these signs have no effect on
divisibility).  

\mni
In particular, the top Chern class of $\text{Sym}^d(Q_{r+1})$ equals the
image of the $\mathfrak{S}_{r+1}$-invariant polynomial,
$$
p_{r+1,d}(u_0,\dots,u_r) = 
\prod_{\underline{d},d_0+\dots+d_r=d} (d_0u_0 + \dots + d_ru_r),
$$
where the product is over all elements $\underline{d}=(d_0,\dots,d_r)$
in $(\ZZ_{\geq 0})^{r+1}$ with $d_0+\dots+d_r = d$.  
In particular, separating out those factors $\underline{d} =
d\mathbf{e}_i$ where only $d_i=d$ and all other $d_j$ are zero, $p_{r+1,d}$
factors as
$$
p_{r+1,d}(u_0,\dots,u_r) = d^{r+1} (u_0u_1\cdots u_r)
q_{r+1,d}(u_0,\dots,u_r),
$$
$$
q_{r+1,d}(u_0,\dots,u_r) = 
\prod_{\underline{d}\neq
  d\mathbf{e}_i, d_0+\dots+d_r=d} (d_0u_0+\dots + d_ru_r).
$$
Thus, the top Chern class of $\text{Sym}^d(S_{r+1}^\vee)$ equals
$d^{r+1}$ times another class, in fact $d^{r+1}c_{r+1}(S_{r+1}^\vee)
\gamma$ where $\gamma$ is the class obtained as the 
image of $q_{r+1,d}(u_0,\dots,u_r)$.  

\mni
A
priori, this top Chern class might be zero as a cycle class.  
However, by Theorem
\ref{thm-Waldron}, when $f_1(n,r,d)\geq 0$,
this class is nonzero: it is Poincar\'{e} dual to
the transversal cycle $\F{1}(r,X)$ for sufficiently general $X$.
Moreover, since $\F{1}(r,X)$ is generically smooth for general $X$, in
the special case that $f_1(n,r,d)$ equals $0$, $\F{1}(r,X)$ is a
zero-dimensional, smooth $k$-scheme whose length equals the degree of
this top Chern class.
By the computation above, the
degree of the top Chern class in $\text{CH}^{(r+1)(n-r)}(\Gg{1}(r,\PP_k(E)))$
equals $d^{r+1}$ times the degree of another cycle.  Thus the length of the
zero-dimensional, smooth $k$-scheme $\F{1}(r,X)$ is divisible by
$d^{r+1}$.  In particular, for $d\geq 2$, $\F{1}(r,X)$ is not
geometrically connected as a $k$-scheme.

\bibliography{my}
\bibliographystyle{alphaurl}

\end{document}